\documentclass[10pt, a4paper]{article} 

\usepackage[utf8]{inputenc} 

\usepackage{graphicx} 
\usepackage[english]{babel}
\usepackage{booktabs} 
\usepackage{array} 
\usepackage{caption}

\usepackage{amsfonts,amsmath,amssymb,bbm,amsthm,amsbsy}
\usepackage[all]{xy}
\SelectTips{cm}{}

\usepackage[pagebackref,
    ,pdfborder={0 0 0}
  ,urlcolor=black,a4paper,hypertexnames=false]{hyperref}

\usepackage{color}
\usepackage{pdfcolmk}

\makeatletter
\def\blfootnote{\xdef\@thefnmark{}\@footnotetext}
\makeatother
\date{\today%
    \protect\blfootnote{\copyright{\ N.~Heuer, C.~L\"oh 2020}. 
    This work was supported by the CRC~1085 \emph{Higher Invariants} 
    (Universit\"at Regensburg, funded by the DFG).
    \\
    MSC~2010 classification: 57N65}}

\usepackage[nouppercase]{scrpage2}
\renewcommand{\headfont}{\sffamily}  
\renewcommand{\pnumfont}{\sffamily}
\automark[subsection]{section}
\usepackage{ifthen}
\deftripstyle{myheadings}
             {\ifthenelse{\isodd{\value{page}}}{\leftmark}{\pagemark}}
             {}
             {\ifthenelse{\isodd{\value{page}}}{\pagemark}{\rightmark}}%
             {}{}{}%
\def\args{\;\cdot\;}

\def\hcomm{%
    \begin{tikzpicture}%
      \draw (0,0) node {\tiny$h$};
      \draw(0,0) circle (0.15);
  \end{tikzpicture}}

\def\varepsilon{\epsilon}
\def\longrightarrow{\to}

\def\phi{\varphi}

\newtheorem{thm}{Theorem}[section]
\newtheorem{lemma}[thm]{Lemma}
\newtheorem{prop}[thm]{Proposition}

\newtheorem{claim}[thm]{Claim}

\theoremstyle{remark}
\newtheorem{rmk}[thm]{Remark}

\theoremstyle{definition}
\newtheorem{defn}[thm]{Definition}
\newtheorem{exmp}[thm]{Example}

\newtheorem{quest}[thm]{Question}

\theoremstyle{theorem}
\newtheorem{theorem}{Theorem}

\usepackage{tikz}
\usepackage{tikz-cd}

\usepackage{IEEEtrantools}

\newenvironment{equ*}[1]{\begin{IEEEeqnarray*}{#1}}{\end{IEEEeqnarray*}}

\newcommand{\R}{\mathbb{R}}

\newcommand{\N}{\mathbb{N}}

\DeclareMathOperator{\SV}{SV}
\DeclareMathOperator{\id}{id}
\DeclareMathOperator{\im}{im}

\DeclareMathOperator{\lf}{lf}
\def\SVlf{\SV^{\lf}}
\def\connsum{\mathbin{\#}}
\def\svlf#1{%
  \| #1 \|^{\lf}}
\def\refet{EQT}
\def\refubc{UBC}

\DeclareMathOperator{\tame}{tame}
\def\SVtame{\SVlf_{\tame}}

\def\qand{\quad\text{and}\quad}

\title{The spectrum of simplicial volume\\ of non-compact manifolds}
\author{Nicolaus Heuer, Clara L\"oh}

\begin{document}

\maketitle

\begin{abstract}
  We show that, in dimension at least~$4$, the set of locally finite
  simplicial volumes of oriented connected open manifolds
  is~$[0,\infty]$. Moreover, we consider the case of tame open
  manifolds and some low-dimensional examples. 
\end{abstract}

\section{Introduction}

Simplicial volumes are invariants of manifolds defined in terms
of the $\ell^1$-semi-norm on singular homology~\cite{vbc}. 

\begin{defn}[simplicial volume]
  Let $M$ be an oriented connected $d$-manifold without boundary.
  Then the \emph{simplicial volume of~$M$} is defined by
  \[ \svlf M := \inf \bigl\{ |c|_1 \bigm| \text{$c \in C_d^{\lf}(M;\R)$ is
                a fundamental cycle of~$M$} \bigr\},
  \]
  where $C^{\lf}_*$ denotes the locally finite singular chain complex. 
  If $M$ is compact, then we also write~$\|M\| := \svlf M$. 
  Using relative fundamental cycles, the notion of simplicial volume can
  be extended to oriented manifolds with boundary. 
\end{defn}

Simplicial volumes are related to negative curvature, volume
estimates, and amenability~\cite{vbc}. In the present
article, we focus on simplicial volumes of \emph{non-compact}
manifolds.  Only few concrete results are known in this context: There
are computations for certain locally symmetric
spaces~\cite{loehsauer,loehsauerhilbert,bucherkimkim,kimkuessner} as well
as the general volume estimates~\cite{vbc}, vanishing
results~\cite{vbc,frigeriomoraschini}, and finiteness
results~\cite{vbc,loeh_l1}.

Let $d \in \N$, let $M(d)$ be the class of all oriented closed
connected $d$-manifolds, and let $M^{\lf}(d)$ be the class of
all oriented connected manifolds without boundary. Then we set
$  \SV(d)
:= \bigl\{ \|M\| \bigm| M \in M(d) \bigr\}
$ and
\begin{align*}
  \SVlf(d)
  & := \bigl\{ \svlf M \bigm| M \in M^{\lf}(d) \bigr\}.
\end{align*}
It is known that $\SV(d)$ is countable and that this set
has no gap at~$0$ if $d \geq 4$:

\begin{thm}[\protect{\cite[Theorem~A]{heuerloeh4mfd}}]\label{thm:nogap}
  Let $d \in \N_{\geq 4}$. Then $\SV(d)$ is dense in~$\R_{\geq 0}$
  and $0 \in \SV(d)$.
\end{thm}

In contrast, if we allow non-compact manifolds, we can realise
\emph{all} non-negative real numbers:

\begin{theorem}\label{theorem:svlfspectrum}
  Let $d\in \N_{\geq 4}$. Then~$\SVlf(d) = [0,\infty]$.
\end{theorem}

The proof uses the no-gap theorem Theorem~\ref{thm:nogap} and
a suitable connected sum construction.

If we restrict to tame manifolds, then we are in a similar situation
as in the closed case:

\begin{theorem}\label{theorem:tame}
  Let $d\in \N$. Then the set $\SVtame(d) \subset [0,\infty]$ is
  countable. In particular, the set~$[0,\infty] \setminus \SVtame(d)$
  is uncountable.
\end{theorem}

As an explicit example, we compute~$\SVlf(2)$ and $\SVtame(2)$
(Proposition~\ref{prop:SVlf2}) as well as~$\SVtame(3)$
(Proposition~\ref{prop:SVtame3}). The case of non-tame
$3$-manifolds seems to be fairly tricky.

\begin{quest}
  What is $\SVlf(3)$\;?
\end{quest}

As $\SV(4) \subset \SVtame(4)$, we know that $\SVtame(4)$ contains
arbitrarily small transcendental numbers~\cite{heuerloehtrans}.

From a geometric point of view, the so-called Lipschitz simplicial volume
is more suitable for Riemannian non-compact manifolds than the locally finite
simplicial volume. It is therefore natural to ask the following: 

\begin{quest}
  Do Theorem~\ref{theorem:svlfspectrum} and Theorem~\ref{theorem:tame}
  also hold for the Lipschitz simplicial volume of oriented connected
  open Riemannian manifolds?
\end{quest}

\subsection*{Organisation of this article}

Section~\ref{sec:proofA} contains the proof of
Theorem~\ref{theorem:svlfspectrum}.  The proof of
Theorem~\ref{theorem:tame} is given in Section~\ref{sec:proofB}.
The low-dimensional case is treated in Section~\ref{sec:lowdim}. 

\section{Proof of Theorem~\ref{theorem:svlfspectrum}}\label{sec:proofA}

Let $d \in \N_{\geq 4}$ and let $\alpha \in [0,\infty]$. Because
$\SV(d)$ is dense in~$\R_{\geq 0}$ (Theorem~\ref{thm:nogap}), there exists a
sequence~$(\alpha_n)_{n \in \N}$ in~$\SV(d)$ with $\sum_{n=0}^\infty
\alpha_n = \alpha$.

\subsection{Construction}

We first describe the construction of a corresponding oriented
connected open manifold~$M$: 
For each~$n \in \N$, we choose an oriented closed connected
$d$-manifold~$M_n$ with~$\| M_n \| = \alpha_n$. Moreover, for~$n >0$, we set
\[ W_n := M_n \setminus (B_{n,-}^\circ \sqcup B_{n,+}^\circ),\]
where $B_{n,-} = i_{n,-}(D^d)$ and $B_{n,+} = i_{n,+}(D^d)$ are two
disjointly embedded closed $d$-balls in~$M_n$. Similarly, we
set~$W_0 := M_0 \setminus B_{0,+}^\circ$. Furthermore, we
choose an orientation-reversing homeomorphism~$f_n \colon S^{d-1} \longrightarrow S^{d-1}$. 
We then consider the infinite ``linear'' connected
sum manifold (Figure~\ref{fig:linearconnsum})
\begin{align*}
  M & := M_0 \connsum M_1 \connsum M_2 \connsum \dots
  \\
    & = (W_0 \sqcup W_1 \sqcup W_n \sqcup \dots)/\!\sim, 
\end{align*}
where $\sim$ is the equivalence relation generated by
\[ i_{n+1,-}(x) \sim i_{n,+}\bigl(f_n(x)\bigr)
\]
for all~$n \in \N$ and all~$x \in S^{d-1} \subset D^d$; we denote the
induced inclusion~$W_n \longrightarrow M$ by~$i_n$.
By construction, $M$ is connected and inherits an orientation
from the~$M_n$.

\begin{figure}
  \begin{center}
    \def\mfdpic#1#2{%
      \begin{scope}[shift={#1}]
        \draw[black!50, dashed] (0,-0.5) -- (0,0.5);
        \draw[black!50, dashed] (2,-0.5) -- (2,0.5);
        \draw (0,0.5) .. controls +(0:0.5) and +(180:0.5) .. (1,1)
                      .. controls +(0:0.5) and +(180:0.5) .. (2,0.5);
        \draw (0,-0.5) .. controls +(0:0.5) and +(180:0.5) .. (1,-1)
                       .. controls +(0:0.5) and +(180:0.5) .. (2,-0.5);
        \draw (1,0) circle (0.3);
        \draw (1,-1.3) node {$#2$};
      \end{scope}}
    \begin{tikzpicture}[thick]
      \draw (0,0.5) .. controls +(180:0.5) and +(90:1.8) .. (-2,0);
      \draw (0,-0.5) .. controls +(180:0.5) and +(-90:1.8) .. (-2,0);
      \draw (-1,0) circle (0.3);
      \draw (-1,-1.3) node {$W_0$};
      \mfdpic{(0,0)}{W_1}
      \mfdpic{(2,0)}{W_2}
      \mfdpic{(4,0)}{W_3}
      \draw (7,0) node {$\dots$};
    \end{tikzpicture}
  \end{center}

  \caption{The construction of~$M$ for the proof of Theorem~\ref{theorem:svlfspectrum}}
  \label{fig:linearconnsum}
\end{figure}
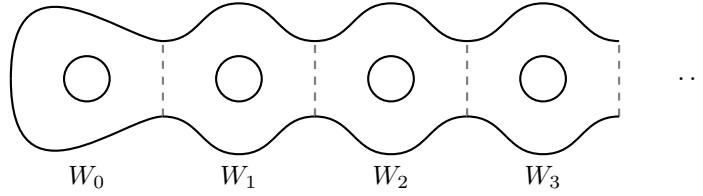

\subsection{Computation of the simplicial volume}

We will now verify that~$\svlf M =\alpha$:

\begin{claim}
  We have~$\svlf M \leq \alpha$. 
\end{claim}
\begin{proof}
  The proof is a straightforward adaption of the chain-level
  proof of sub-additivity of simplicial volume with respect
  to amenable glueings. 
  
  In particular, we will use the uniform boundary
  condition~\cite{matsumotomorita}  and the
  equivalence theorem~\cite{vbc,bbfipp}:
  \begin{itemize}
  \item[UBC]\label{ubc}
    The chain complex~$C_*(S^{d-1};\R)$
    satisfies~$(d-1)$-UBC, i.e., there is a constant~$K$ such that:
    For each~$c \in \im \partial_d \subset C_{d-1}(S^{d-1};\R)$, there
    exists a chain~$b \in C_d(S^{d-1};\R)$ with
    \[ \partial_d b = c \qand |b|_1 \leq K \cdot |c|_1.
    \]
  \item[EQT]\label{et}
    Let $N$ be an oriented closed connected $d$-manifold, let $B_1, \dots, B_k$
    be disjointly embedded $d$-balls in~$N$, and
    let~$W := N \setminus (B_1^\circ \cup \dots, \cup B_1^\circ)$. Moreover,
    let $\varepsilon \in \R_{>0}$.
    Then
    \[ \|N\| = \inf \bigl\{ |z|_1 \bigm| z \in Z(W;\R),\ |\partial_d z|_1 \leq \varepsilon \bigr\},
    \]
    where $Z(W;\R) \subset C_d(W;\R)$ denotes the set of all relative
    fundamental cycles of~$W$.
  \end{itemize}

  Let $\varepsilon \in \R_{>0}$. 
  By~\refet, for each~$n \in\N$, there exists a relative fundamental cycle~$z_n \in Z(W_n;\R)$
  with
  \[ |z_n|_1 \leq \alpha_n + \frac1{2^n}\cdot \varepsilon
  \qand
     |\partial_d z_n|_1 \leq \frac1{2^n} \cdot \varepsilon.
  \]

  We now use~\refubc\ to construct a locally finite fundamental cycle
  of~$M$ out of these relative cycles: For~$n \in \N$, the boundary
  parts  
  $C_{d-1}(i_n;\R)(\partial_d z_n|_{B_{n,+}})$ and~$-C_{d-1}(i_{n+1};\R)(\partial_d
  z_{n+1}|_{B_{n+1},-})$ are fundamental cycles of the sphere~$S^{d-1}$
  (embedded via~$i_n \circ i_{n,+}$ and $i_{n+1} \circ i_{n+1,-}$ 
  into~$M$, which implicitly uses the orientation-reversing
  homeomorphism~$f_n$). 
  By~\refubc, 
  there exists a chain~$b_n \in C_d(S^{d-1};\R)$ with
  \begin{align*}
    \partial_d C_d(i_n \circ i_{n,+};\R) (b_n)
     = &\; C_{d-1}(i_n ;\R)(\partial_d z_n|_{B_{n,+}})\\
     + &\; C_{d-1}(i_{n+1};\R) (\partial_d z_{n+1}|_{B_{n+1,-}})
  \end{align*}
  and
  \[
    |b_n|_1 \leq K \cdot \Bigl(\frac1{2^n} + \frac1{2^{n+1}}\Bigr) \cdot \varepsilon
            \leq K \cdot \frac1{2^{n-1}} \cdot \varepsilon.
  \]
  A straightforward computation shows that
  \[ c := \sum_{n=0}^\infty C_d(i_n;\R)\bigl(z_n - C_d(i_{n,+};\R)(b_n)\bigr)
  \]
  is a locally finite $d$-cycle on~$M$. Moreover, the local contribution on~$W_0$
  shows that $c$ is a locally finite fundamental cycle of~$M$.
  By construction,
  \begin{align*}
    |c|_1 & \leq \sum_{n=0}^\infty \bigl(|z_n|_1 + |b_n|_1\bigr)
    \\
    & \leq \sum_{n=0}^\infty \Bigl( \alpha_n + \frac1{2^n} \cdot \varepsilon
                               + K \cdot \frac1{2^{n-1}} \cdot \varepsilon
                               \Bigr)
    \leq \sum_{n=0}^\infty \alpha_n + (2 + 4 \cdot K) \cdot \varepsilon
    \\
    & = \alpha + (2 + 4 \cdot K) \cdot \varepsilon.
  \end{align*}
  Thus, taking~$\varepsilon \rightarrow 0$, we obtain~$\svlf M\leq \alpha$.
\end{proof}

\begin{claim}\label{claim:geq}
  We have~$\svlf M \geq \alpha$.
\end{claim}
\begin{proof}
  Without loss of generality we may assume that $\svlf M$ is finite.
  Let $c \in C^{\lf}_d(M;\R)$ be a locally finite fundamental cycle
  of~$M$ with $|c|_1 < \infty$. For~$n \in \N$, we
  consider the subchain~$c_n := c|_{W_{(n)}}$ of~$c$, consisting of all simplices
  whose images touch~$W_{(n)} := \bigcup_{k=0}^n i_k(W_k) \subset M$.
  Because $c$ is locally finite, each~$c_n$ is a finite singular chain
  and $(|c_n|_1)_{n\in \N}$ is a monotonically increasing sequence
  with limit~$|c|_1$.

  Let $\varepsilon \in \R_{>0}$. Then there is an~$n \in \N_{>0}$ that
  satisfies 
  $|c-c_n|_1 \leq \varepsilon$ and $\alpha - \sum_{k=0}^n \alpha_k \leq \varepsilon$.
  Let
  \[ p \colon M \longrightarrow W_{(n)}/i_n(B_{n,+}) =: W
  \]
  be the map that collapses everything beyond stage~$n+1$ to a single
  point~$x$.
  Then $z := C_d(p;\R)(c_n) \in C_d(W,\{x\};\R)$ is a relative cycle
  and 
  \[ |\partial_d z|_1 \leq |\partial_d c_n|_1
  \leq |\partial_d (c-c_n)|_1
  \leq (d+1) \cdot |c- c_n|_1
  \leq (d+1) \cdot \varepsilon.
  \]
  Because~$d > 1$, there exists a chain~$b \in C_{d}(\{x\};\R)$ with
  \[ \partial_{d} b = \partial_d z
     \qand |b|_1 \leq |\partial_d z| \leq (d+1) \cdot \varepsilon.
  \]
  Then
  \[ \overline z := z - b \in C_d(W;\R)
  \]
  is a cycle on~$W$; because $z$ and $\overline z$ have the
  same local contribution on~$W_0$, the cycle~$z$ is a fundamental
  cycle of the manifold
  \[ W \cong M_0 \connsum \dots \connsum M_n.
  \]
  As $d > 2$, the construction of our chains and
  additivity of simplicial volume under connected sums~\cite{vbc,bbfipp}
  show that
  \begin{align*}
    |c|_1
    & \geq |c_n|_1
    \geq |z|_1
    \geq |\overline z|_1 - |b|_1
    \\
    & \geq \|W\| - (d+1) \cdot \varepsilon
    = \sum_{k=0}^n \|M_n\| - (d+1) \cdot \varepsilon
    \\
    & \geq \alpha - (d+2) \cdot \varepsilon.
  \end{align*}
  Thus, taking~$\varepsilon \rightarrow 0$, we obtain~$|c|_1 \geq
  \alpha$; hence, $\svlf M \geq
  \alpha$.
\end{proof}

This completes the proof of Theorem~\ref{theorem:svlfspectrum}.

\begin{rmk}[adding geometric structures]
  In fact, this argument can also be performed smoothly: The
  constructions leading to Theorem~\ref{thm:nogap} can be carried out in
  the smooth setting. Therefore, we can choose the~$(M_n)_{n \in \N}$
  to be smooth and equip~$M$ with a corresponding smooth
  structure. Moreover, we can endow these smooth pieces with Riemannian
  metrics. Scaling these Riemannian metrics appropriately shows that we
  can turn~$M$ into a Riemannian manifold of finite volume. 
\end{rmk}
  
\section{Proof of Theorem~\ref{theorem:tame}}\label{sec:proofB}

In this section, we prove Theorem~\ref{theorem:tame}, i.e., that the
set of simplicial volumes of tame manifolds is countable. 

\begin{defn}
  A manifold~$M$ without boundary is \emph{tame} if there exists a
  compact connected manifold~$W$ with boundary such that $M$ is
  homeormorphic to $W^\circ := W \setminus \partial W$.
\end{defn}

As in the closed case, our proof is based on a counting argument:

\begin{prop}\label{prop:counttame}
  There are only countably many proper homotopy types of tame manifolds.
\end{prop}

As we could not find a proof of this statement in the literature, we
will give a complete proof in Section~\ref{subsec:counttame} below. 
Theorem~\ref{theorem:tame} is a direct consequence of
Proposition~\ref{prop:counttame}:

\begin{proof}[Proof of Theorem~\ref{theorem:tame}]
  The simplicial volume~$\svlf{\cdot}$ is invariant under proper
  homotopy equivalence (this can be shown as in the compact case).
  Therefore, the countability of~$\SVlf(d)$ follows from the
  countability of the set of proper homotopy types of tame
  $d$-manifolds (Proposition~\ref{prop:counttame}).
\end{proof}

\begin{rmk}\label{rem:tameinfty}
  Let $d \in \N_{\geq 3}$. Then $\infty \in \SVtame(d)$: Let $N$
  be an oriented closed connected hyperbolic $(d-1)$-manifold
  and let $M := N \times \R$. Then $M$ is tame (as interior of~$N \times [0,1]$)
  and $\|N\| > 0$~\cite[Section~0.3]{vbc}\cite[Theorem~6.2]{thurston}. Hence, 
  by the finiteness criterion~\cite[p.~17]{vbc}\cite[Theorem~6.4]{loeh_l1},
  we obtain that $\svlf M = \infty$. 
\end{rmk}

\subsection{Counting tame manifolds}\label{subsec:counttame}

It remains to prove Proposition~\ref{prop:counttame}. We use the
following observations:

\begin{defn}[models of tame manifolds]
  \hfil
  \begin{itemize}
  \item A \emph{model} of a tame manifold~$M$ is a finite CW-pair~$(X,A)$ (i.e.,
    a finite CW-complex~$X$ with a finite subcomplex~$A$) that
    is homotopy equivalent (as pairs of spaces) to~$(W,\partial W)$,
    where $W$ is a compact connected manifold with boundary whose
    interior is homeomorphic to~$M$.
  \item Two models of tame manifolds are \emph{equivalent} if they
    are homotopy equivalent as pairs of spaces.
  \end{itemize}
\end{defn}

\begin{lemma}[existence of models]\label{lem:modelsexist}
  Let $W$ be a compact connected manifold. Then there exists a finite
  CW-pair~$(X,A)$ such that $(W,\partial W)$ and $(X,A)$ are homotopy
  equivalent pairs of spaces.

  In particular: Every tame manifold admits a model.
\end{lemma}
\begin{proof}
  It should be noted that we work with topological manifolds; hence,
  we cannot argue directly via triangulations. Of course, the main
  ingredient is the fact that every compact manifold is homotopy equivalent
  to a finite complex~\cite{siebenmann,kirbysiebenmann}.

  Hence, there exist finite CW-complexes~$A$ and~$Y$ with homotopy
  equivalences~$f \colon A\longrightarrow \partial W$ and $g \colon Y
  \longrightarrow W$. Let $j := \overline g \circ i \circ f$, where $i \colon
  \partial W \hookrightarrow W$ is the inclusion and $\overline g$
  is a homotopy inverse of~$g$. By construction, the upper square in
  the diagram in Figure~\ref{fig:modelsexist} is homotopy commutative.

  As next step, we replace~$j \colon A \longrightarrow Y$ by a
  homotopic map~$j_c \colon A \longrightarrow Y$ that is cellular
  (second square in Figure~\ref{fig:modelsexist}).

  The mapping cylinder~$Z$ of~$j_c$ has a finite CW-structure (as
  $j_c$ is cellular) and the canonical map~$p \colon Z \longrightarrow
  Y$ allows to factor~$j_c$ into an inclusion~$J$ of a subcomplex and
  the homotopy equivalence~$p$ (third square in
  Figure~\ref{fig:modelsexist}).

  We thus obtain a homotopy commutative square
  \[ \xymatrix{%
    \partial W \ar@{}[dr] |{\hcomm}
    \ar[r]^-{i}
    & W
    \\
    A \ar[r]^-{J} \ar[u]^-{f}
    & Z\ar[u]_-{F := g \circ p}
    }
  \]
  where the vertical arrows are homotopy equivalences,
  the upper horizontal arrow is the inclusion, and the
  lower horizontal arrow is the inclusion of a subcomplex.

  Using a homotopy between~$i \circ f$ and $F \circ J$ and
  adding another cylinder to~$Z$, we can replace~$Z$ by
  a finite CW-complex~$X$ (that still contains~$A$ as
  subcomplex) to obtain a \emph{strictly} commutative diagram
  \[ \xymatrix{%
    \partial W \ar[r]^-{i}
    & W
    \\
    A \ar[u]_-{\simeq}^-{f} \ar[r]
    & X \ar[u]^-{\simeq}
    }
  \]
  whose vertical arrows are homotopy equivalences and whose
  horizontal arrows are inclusions.

  Because the inclusions~$\partial W \hookrightarrow W$
  (as inclusion of the boundary of a compact
  topological manifold) 
  and $A \hookrightarrow X$ (as inclusion of a subcomplex)
  are cofibrations, this already implies that the vertical
  arrows form a homotopy equivalence~$(X,A) \longrightarrow (W,\partial W)$
  of pairs~\cite[Chapter~6.5]{mayconcise}.
\end{proof}

\begin{figure}
  \begin{align*}
    \xymatrix{%
      \partial W \ar @{} [dr] |{\hcomm}
      \ar[r]^-{i}
      & W \ar@{-->}@<1mm>[d]^-{\overline g}
      \\
      A \ar[u]^-{f} \ar[r]^-{j} \ar@{}[dr]|{\hcomm}
      & Y \ar@<1mm>[u]^-{g}
      \\
      A \ar@{=}[u] \ar[r]^-{j_c} 
      & Y \ar@{=}[u]
      \\
      A \ar@{=}[u] \ar[r]^-{J}
      & Z \ar[u]_-{p}
    }
  \end{align*}

  \caption{Finding a model}
  \label{fig:modelsexist}
\end{figure}

\begin{lemma}[equivalence of models]\label{lem:equivalentmodels}
  If $M$ and $N$ are tame manifolds with equivalent models, then
  $M$ and $N$ are properly homotopy equivalent.
\end{lemma}
\begin{proof}
  As $M$ and $N$ admit equivalent models, there exist compact
  connected manifolds~$W$ and $V$ with boundary such that $M \cong
  W^\circ$ and $N \cong V^\circ$ and such that the pairs $(W,\partial
  W)$ and $(V,\partial V)$ are homotopy equivalent (by transitivity of
  homotopy equivalence of pairs of spaces). Let
  $(f, f_\partial) \colon (W,\partial W) \longrightarrow (V,\partial V)$
  and $(g,g_\partial) \colon (V,\partial V) \longrightarrow (W,\partial W)$
  be mutually homotopy inverse homotopy equivalences of pairs.

  By the topological collar
  theorem~\cite{brown_collar,connelly_collar}, we have homeomorphisms
  \begin{align*}
    M & \cong W \cup_{\partial W} \bigl( \partial W \times[0,\infty) \bigr)
    \\
    N & \cong V \cup_{\partial V} \bigl( \partial V \times[0,\infty) \bigr),  
  \end{align*}
  where the glueing occurs via the canonical inclusions~$\partial W
  \hookrightarrow \partial W \times [0,\infty)$ and $\partial V
    \hookrightarrow \partial V \times [0,\infty)$ at parameter~$0$.
    
  Then the maps~$f$ and $f_\partial \times \id_{[0,\infty)}$ glue to
  a well-defined proper continuous map~$F \colon M \longrightarrow N$
  and the maps~$g$ and $g_\partial \times \id_{[0,\infty)}$ glue to
  a well-defined proper continuous map~$G \colon N \longrightarrow M$.     

  Moreover, the homotopy of pairs between~$(f \circ g, f_\partial \circ
  g_\partial)$ and~$(\id_V,\id_{\partial V})$ glues into a proper
  homotopy between~$F \circ G$ and~$\id_M$. In the same way, there
  is a proper homotopy between~$G \circ F$ and~$\id_N$. Hence, the
  spaces~$M$ and $N$ are properly homotopy equivalent.
\end{proof}

\begin{lemma}[countability of models]\label{lem:countmodels}
  There exist only countably many equivalence classes of models.
\end{lemma}
\begin{proof}
  There are only countably many homotopy types of finite CW-complexes
  (because every finite CW-complex is homotopy equivalent to a finite
  simplicial complex). Moreover, every finite CW-complex
  has only finitely many subcomplexes. Therefore, there are only countably
  many homotopy types (of pairs of spaces) of finite CW-pairs.
\end{proof}

\begin{proof}[Proof of Proposition~\ref{prop:counttame}]
  We only need to combine Lemma~\ref{lem:modelsexist},
  Lemma~\ref{lem:equivalentmodels}, and Lemma~\ref{lem:countmodels}.
\end{proof}

\section{Low dimensions}\label{sec:lowdim}

\subsection{Dimension~$2$}

We now compute the set of simplicial volumes of surfaces. We first
consider the tame case:

\begin{exmp}[tame surfaces]\label{exa:tamesurf}
  Let $W$ be an oriented compact connected surface with~$g\in \N$ handles
  and $b \in \N$ boundary components. Then the proportionality principle
  for simplicial volume of hyperbolic manifolds~\cite[p.~11]{vbc}
  (a thorough exposition is given, for instance, by Fujiwara and
  Manning~\cite[Appendix~A]{fujiwaramanning}) gives 
  \[
    \svlf{W^\circ} =
    \begin{cases}
      4 \cdot (g-1) + 2 \cdot b & \text{if~$g>0$}
      \\
      2 \cdot b - 4 & \text{if $g=0$ and $b>1$}
      \\
      0 & \text{if $g=0$ and $b \in \{0,1\}$.}
    \end{cases}
  \]  
\end{exmp}

\begin{prop}\label{prop:SVlf2}
  We have~$\SVlf(2) = 2 \cdot \N \cup \{\infty\}$
  and $\SVtame(2) = 2 \cdot \N$.
\end{prop}
\begin{proof}
  We first prove~$2 \cdot \N \subset \SVtame(2) \subset \SVlf(2)$ and
  $\infty \in \SVlf(2)$, i.e., that all the given values may be
  realised: In view of Example~\ref{exa:tamesurf}, all even numbers
  occur as simplicial volume of some (possibly open) tame surface.
  
  Let
  \[ M := T^2 \connsum T^2 \connsum T^2 \connsum \dots
  \]
  be an infinite ``linear'' connected sum of tori~$T^2$. Collapsing~$M$ to
  the first~$g \in \N$ summands and an argument as in the proof of Claim~\ref{claim:geq}
  shows that
  \[ \svlf M \geq \| \Sigma_g \| = 4 \cdot g - 4
  \]
  for all~$g \in \N_{\geq 1}$. Hence, $\svlf M = \infty$. 
  
  It remains to show that $\SVlf(2) \subset 2 \cdot \N \cup \{\infty\}$:
  Let $M$ be an oriented connected (topological, separable, Hausdorff)
  $2$-manifold without boundary.
  Then $M$ admits a smooth structure~\cite{moise} and whence a proper
  smooth map~$p \colon M \longrightarrow \R$. Using suitable regular
  values of~$p$, we can thus write~$M$ as an ascending union
  \[ M = \bigcup_{n\in \N} M_n
  \]
  of oriented connected compact submanifolds (possibly with boundary)~$M_n$
  that are nested via~$M_0 \subset M_1 \subset \dots$. Then one of the following
  cases occurs:
  \begin{enumerate}
  \item\label{i:tame} There exists an~$N \in \N$ such that for all~$n
    \in \N_{\geq N}$ the inclusion~$M_n \hookrightarrow M_{n+1}$ is a
    homotopy equivalence.
  \item\label{i:nontame} For each~$N \in \N$ there exists an~$n \in
    \N_{\geq N}$ such that the inclusion~$M_n \hookrightarrow M_{n+1}$
    is \emph{not} a homotopy equivalence.
  \end{enumerate}

  In the first case, the classification of compact surfaces with
  boundary shows that $M$ is tame. Hence~$\svlf M \in 2 \cdot \N$
  (Example~\ref{exa:tamesurf}).

  In the second case, the manifold~$M$ is \emph{not} tame (which
  can, e.g., be derived from the classification of compact surfaces
  with boundary). We show that $\svlf M = \infty$. To this end. 
  we distinguish two cases:
  \begin{enumerate}
    \renewcommand{\labelenumi}{\alph{enumi}.}
  \item The sequence~$(h(M_n))_{n \in \N}$ is unbounded, where $h(\args)$
    denotes the number of handles of the surface.
  \item The sequence~$(h(M_n))_{n \in \N}$ is bounded.
  \end{enumerate}

  In the unbounded case, a collapsing argument (similar to the
  argument for~$T^2 \connsum T^2 \connsum \dots$ and Claim~\ref{claim:geq})
  shows that $\svlf M = \infty$.

  We claim that also in the bounded case we have~$\svlf M = \infty$: Shifting
  the sequence in such a way that all handles are collected in~$M_0$, we may
  assume without loss of generality that the sequence~$(h(M_n))_{n \in \N}$ is
  constant. Thus, for each~$n \in \N$, the surface~$M_{n+1}$ is obtained from~$M_n$
  by adding a finite disjoint union of disks and of spheres with finitely many
  (at least two) disks removed; we can reorganise this sequence in such a way
  that no disks are added. Hence, we may assume that $M_{n}$ is a retract of~$M_{n+1}$
  for each~$n \in \N$. 
  Furthermore, because we are in case~\ref{i:nontame}, the classification of compact
  surfaces shows (with the help of Example~\ref{exa:tamesurf}) that
  \[ \lim_{n\rightarrow \infty} \|M_n\| = \infty.
  \]
  
  Let $c \in C^{\lf}_2(M;\R)$ be a locally finite fundamental cycle
  of~$M$ and let $n\in \N$.  Because $c$ is locally finite, there is
  a~$k \in \N$ such that~$c|_{M_n}$ is supported on~$M_{n+k}$; the
  restriction~$c|_{M_n}$ consists of all summands of~$c$ whose supports
  intersect with~$M_n$. Because $M_n$ is a retract of~$M_{n+k}$, we
  obtain from~$c|_{M_n}$ a relative fundamental cycle~$c_n$ of~$M_n$
  by pushing the chain~$c|_{M_n}$ to~$M_n$ via a retraction~$M_{n+k}
  \longrightarrow M_n$. Therefore,
  \[ |c|_1 \geq |c|_{M_n}|_1 \geq |c_n|_1 \geq \|M_n\|.
  \]
  Taking~$n \rightarrow \infty$ shows that $|c|_1 = \infty$. Taking
  the infimum over all locally finite fundamental cycles~$c$ of~$M$ proves
  that~$\svlf M = \infty$.

  Moreover, Example~\ref{exa:tamesurf} shows that $\infty\not\in \SVtame(2)$.
\end{proof}

\subsection{Dimension~$3$}

The general case of non-compact $3$-manifolds seems to be rather
involved (as the structure of non-compact $3$-manifolds can get
fairly complicated). We can at least deal with the tame case:

\begin{prop}\label{prop:SVtame3}
  We have~$\SVtame(3) = \SV(3) \cup \{\infty\}$.
\end{prop}
\begin{proof}
  Clearly, $\SV(3) \subset \SVtame(3)$ and $\infty \in \SVtame(3)$
  (Remark~\ref{rem:tameinfty}).

  Conversely, let $W$ be an oriented compact connected $3$-manifold
  and let $M := W^\circ$. We distinguish the following cases:
  \begin{itemize}
  \item
    If at least one of the boundary components of~$W$ has genus at least~$2$, then
    the finiteness criterion~\cite[p.~17]{vbc}\cite[Theorem~6.4]{loeh_l1}
    shows that $\svlf M = \infty$.
  \item
    If the boundary of~$W$ consists only of spheres and tori, then we
    proceed as follows: In a first step, we fill in all spherical
    boundary components of~$W$ by $3$-balls and thus obtain an oriented
    compact connected $3$-manifold~$V$ all of whose boundary
    components are tori. In view of considerations on tame manifolds
    with amenable boundary~\cite{kimkuessner} and glueing
    results for bounded
    cohomology~\cite{vbc}\cite{bbfipp},
    we obtain that
    \[ \svlf M = \| W \| = \|V\|.
    \]
    By Kneser's prime decomposition theorem~\cite[Theorem~1.2.1]{afw}
    and the additivity of (relative) simplicial volume with respect to
    connected
    sums~\cite{vbc}\cite{bbfipp} in dimension~$3$, we may
    assume that $V$ is prime (i.e., admits no non-trivial
    decomposition as a connected sum). Moreover, because~$\|S^1 \times S^2\| = 0$,
    we may even assume that $V$ is irreducible~\cite[p.~3]{afw}.
    
    By geometrisation~\cite[Theorem~1.7.6]{afw}, then $V$ admits a
    decomposition along finitely many incompressible tori into Seifert
    fibred manifolds (which have trivial simplicial
    volume~\cite[Corollary~6.5.3]{thurston})  and hyperbolic pieces~$V_1, \dots, V_k$. As
    the tori are incompressible, we can now again apply
    additivity~\cite{vbc}\cite{bbfipp} to
    conclude that
    \[ \|V\| = \sum_{j=1}^k \|V_j\|.
    \]

    Let $j \in \{1,\dots, k\}$. Then the boundary components of~$V_j$
    are $\pi_1$-injective tori (as the interior of~$V_j$ admits a
    complete hyperbolic metric of finite volume)~\cite[Proposition~D.3.18]{benedettipetronio}. Let $S$
    be a Seifert $3$-manifold whose boundary is a $\pi_1$-injective
    torus (e.g., the knot complement of a non-trivial torus
    knot~\cite[Theorem~2]{moser}\cite[Lemma~4.4]{lueckl2}).
    Filling each boundary component
    of~$V_j$ with a copy of~$S$ results in an oriented closed
    connected $3$-manifold~$N_j$, which satisfies (again, by
    additivity)
    \[ \|N_j\| = \|V_j\| + 0 = \|V_j\|.
    \]
    Therefore, the oriented closed connected $3$-manifold~$N := N_1
    \connsum \dots \connsum N_k$ satisfies
    \[ \|N\| = \sum_{j=1}^k \|N_j\| = \sum_{j=1}^k \|V_j\| = \|V\|.
    \]
    In particular, $\svlf M = \|V\| = \|N\| \in \SV(3)$. 
    \qedhere
  \end{itemize}
\end{proof}

{\small
\bibliographystyle{alpha}
\bibliography{bib_l1}}

\vfill

\noindent
\emph{Nicolaus Heuer}\\[.5em]
  {\small
  \begin{tabular}{@{\qquad}l}
DPMMS,    University of Cambridge \\
    \textsf{nh441@cam.ac.uk},
    \textsf{https://www.dpmms.cam.ac.uk/$\sim$nh441}
  \end{tabular}}

\medskip

\noindent
\emph{Clara L\"oh}\\[.5em]
  {\small
  \begin{tabular}{@{\qquad}l}
    Fakult\"at f\"ur Mathematik,
    Universit\"at Regensburg,
    93040 Regensburg\\
    \textsf{clara.loeh@mathematik.uni-r.de}, 
    \textsf{http://www.mathematik.uni-r.de/loeh}
  \end{tabular}}

\end{document}